\documentclass[12pt]{article}
\usepackage[T2A]{fontenc}
\usepackage[cp1251]{inputenc}
\usepackage[english]{babel}
\usepackage{amsmath,latexsym,amsthm,amsfonts,tipa,upgreek}
\usepackage{amssymb,amscd,graphpap}
\newcommand\NN{{\mathbb N}}
\newcommand\ZZ{{\mathbb Z}}

\newcommand\w{{\omega}}
\newcommand\vt{{\Updelta}}

\newcommand\II{{\mathcal I}}
\newcommand\PP{{\mathcal P}}
\newcommand\FF{{\mathcal F}}

\newtheorem{Qs}{Question}

\newtheorem{Ps}{Proposition}

\begin{document}

\title{Ultracompanions of subsets of a group}
\author{I.~Protasov, S.~Slobodianiuk}
\date{}
\maketitle

\begin{abstract}
Let $G$ be a group, $\beta G$ is the Stone-$\check{C}$ech compactification of $\beta G$ endowed with the structure of a right topological semigroup, $G^*=\beta G\setminus G$. Given any subset $A$ of $G$ and $p\in G^*$, we define the $p$-companion $\vt_p(A)=A^*\cap Gp$ of $A$, and characterize the subsets with finite and discrete ultracompanions.

\

{\bf 2010 AMS Classification}: 54D35, 22A15, 20F69.

\

\end{abstract}

\section{Introduction}

Given a discrete space $X$, we take the points of $\beta X$, the Stone-$\check{C}$ech compactification of $X$, to be the ultrafilters on $X$, with the points of $X$ identified with the principal ultrafilters, so $X^*=\beta X\setminus  X$ is the set of all free ultrafilters on $X$. The topology on $\beta X$ can be defined by stating that the sets of the form $\overline{A}=\{p\in\beta X: A\in p\}$, where $A$ is a subset of $X$, are base for the open sets. We note the sets of this form are clopen and that for any $p\in\beta X$ and $A\subseteq X$, $A\in p$ if and only if $p\in\overline{A}$. For any $A\subseteq X$, we denote $A^*=\overline{A}\cap G^*$. The universal property of $\beta X$ states that every mapping $f: X\to Y$, where $Y$ is a compact Hausdorff space, can be extended to the continuous mapping $f^\beta:\beta X\to Y$.

Now let $G$ be a discrete group. Using the universal property of $\beta G$, we can extend the group multiplication from $G$ to $\beta G$ in two steps. Given $g\in G$, the mapping $$x\mapsto gx: \text{  } G\to \beta G$$
extends to the continuous mapping $$q\mapsto gq: \text{ } \beta G\to \beta G.$$
Then, for each $q\in\beta G$, we extend the mapping $g\mapsto gq$ defined from $G$ into $\beta G$ to the continuous mapping $$p\mapsto pq:\text{ }\beta G\to\beta G.$$
The product $pq$ of the ultrafilters $p$, $q$ can also be defined by the rule: given a subset $A\subseteq G$, $$A\in pq\leftrightarrow\{g\in G:g^{-1}A\in q\}\in p.$$
To describe a base for $pq$, we take any element $P\in p$ and, for every $x\in P$, choose some element $Q_x\in q$. Then $\bigcup_{x\in P}xQ_x\in pq$, and the family of subsets of this form is a base for the ultrafilter $pq$.

By the construction, the binary operation $(p,q)\mapsto pq$ is associative, so $\beta G$ is a semigroup, and $G^*$ is a subsemigroup of $\beta G$. For each $q\in \beta G$, the right shift $x\mapsto xq$ is continuous, and the left shift $x\to gx$ is continuous for each $g\in G$.

For the structure of a compact right topological semigroup $\beta G$ and plenty of its applications to combinatorics, topological algebra and functional analysis see ~\cite{b2}, ~\cite{b4}, ~\cite{b5}, ~\cite{b18}, ~\cite{b21}.

Given a subset $A$ of a group $G$ and an ultrafilter $p\in G^*$ we define a {\em $p$-companion} of $A$ by
$$\vt_p(A)=A^*\cap Gp=\{gp: g\in G, A\in gp\},$$
and say that a subset $S$ of $G^*$ is an {\em ultracompanion} of $A$ if $S=\vt_p(A)$ for some $p\in G^*$.

Clearly,$A$ is finite if and only if $\vt_p(A)=\varnothing$ for every $p\in G^*$, and $\vt_p(G)=Gp$ for each $p\in G^*$.

We say that a subset$A$ of a group $G$ is
\\$\bullet$ {\em sparse} if each ultracompanion of $A$ is finite;
\\$\bullet$ {\em disparse} if each ultracompanion of $A$ is discrete.

In fact, the sparse subsets were introduced in ~\cite{b3} with rather technical definition (see Proposition 5) in order to characterize strongly prime ultrafilters in $G^*$, the ultrafilters from $G^*\setminus\overline{G^*G^*}$.

In this paper we study the families of sparse and disparse subsets of a group, and characterize in terms of ultracompanions the subsets from the following  basic classification.

A subset $A$ of $G$ is called
\\$\bullet$ {\em large} if $G=FA$ for some finite subset $F$ of $G$;
\\$\bullet$ {\em thick} if, for every finite subset $F$ of $G$, there exists $a\in A$ such that $Fa\subseteq A$;
\\$\bullet$ {\em prethick} if $FA$ is thick for some finite subset $F$ of $G$;
\\$\bullet$ {\em small} if $L\setminus A$ is large for every large subset $L$;
\\$\bullet$ {\em thin} if $gA\cap A$ is finite for each $g\in G\setminus\{e\}$, $e$ is the identity of $G$.

In the dynamical terminology ~\cite{b5}, the large and prethick subsets are called syndetic and piecewise syndetic respectively. For references on the subset combinatorics of groups see the survey ~\cite{b12}.

We conclude the paper with discussions of some modifications of sparse subsets and a couple of open questions.

\section{Characterizations}~\label{s2}
\begin{Ps}~\label{p1} For a subset $A$ of a group $G$ and an ultrafilter $p\in G^*$, the following statements hold

$(i)$ $\vt_p(FA)=F\vt_p(A)$ for every finite subset $F$ of $G$;

$(ii)$ $\vt_p(Ah)=\vt_{ph^{-1}}(A)$ for every $h\in G$;

$(iii)$ $\vt_p(A\cup B)=\vt_p(A)\cup\vt_p(B)$.\end{Ps}
\begin{Ps} For an infinite subset $A$ of a group $G$, the following statements are equivalent

$(i)$ $A$ is large;

$(ii)$ there exists a finite subset $F$ of $G$ such that, for each $p\in G^*$, we have $Gp=\vt_p(FA)$;

$(iii)$ for every $p\in G^*$, there exists a finite subset $F_p$ of $G$ such that $Gp=\vt_p(F_pA)$.\end{Ps}
\begin{proof} The implications $(i)\Rightarrow (ii)\Rightarrow(iii)$ are evident. To prove $(iii)\Rightarrow(i)$, we note that the family $\{(F_pA)^*:p\in G\}$ is a covering of $G^*$, choose a finite subcovering $(F_{p_1}A)^*,...,(F_{p_n}A)^*$ and put $F=F_{p_1}\cup...\cup F_{p_n}$. Then $G^*=(FA)^*$ so $G\setminus FA$ is finite and $G=HFA$ for some finite subset $H$ of $G$. Hence, $A$ is large. \end{proof}
\begin{Ps} For an infinite subset $A$ of a group $G$ the following statements hold

$(i)$ $A$ is thick if and only if there exists $p\in G^*$ such that $\vt_p(A)=Gp$;

$(ii)$ $A$ is prethick if and only if there exists $p\in G^*$ and a finite subset $F$ of $G$ such that $\vt_p(FA)=Gp$;

$(iii)$ $A$ is small if and only if, for every $p\in G^*$ and each finite $F$ of $G$, we have $\vt_p(FA)\neq Gp$;

$(iv)$ $A$ is thin if and only if $|\vt_p(A)|\leq1$ for each $p\in G^*$.\end{Ps}
\begin{proof}~\label{p3} $(i)$ Assume that $A$ is thick. For each finite subset $F$ of $G$, we put $P_F=\{x\in A:Fx\subset A\}$ and form a family $\PP=\{P_F:F  \text{ is a finite subset of } $G$\}$. Since $A$ is thick, each subset $P_F$ is infinite. Clearly, $P_F\cap P_H=P_{P\cup H}$. Therefore, $\PP$ is contained in some ultrafilter $p\in G^*$. By the choice of $\PP$, we have, $\vt_p(A)=Gp$.

On the other hand, let $\vt_p(A)=Gp$. We take an arbitrary finite subset $F$ of $G$. Then $(F\cup\{e\})p\subset A^*$ so $(F\cup\{e\})P\subset A$ for some $P\in p$. Hence, $P\subseteq A$ and $Fx\subset A$ for each $x\in P$.

$(ii)$ follows from $(i)$.

$(iii)$ We note that $A$ is small if and only if $A$ is not prethick and apply $(ii)$.

$(iv)$ follows directly from the definitions of thin subsets and $\vt_p(A)$. \end{proof}

For $n\in\NN$, a subset $A$ of a group $G$ is called $n$-thin if, for every finite subset $F$ of $G$, there is a finite subset $H$ of $G$ such that $|Fg\cap A|\le n$ for every $g\in G\setminus H$.
\begin{Ps} For a subset $A$ of a group $G$, the following statements are equivalent

$(i)$ $|\vt_p(A)|\le n$ for each $p\in G^*$;

$(ii)$ for every distinct $x_1,..,x_{n+1}\in G$, the set $x_1A\cap...\cap x_{n+1}A$ is finite;

$(iii)$ $A$ is $n$-thin. \end{Ps}
\begin{proof} We note that $x_1A\cap...\cap x_{n+1}A$ is infinite if and only if there exists $p\in G^*$ such that $x_1^{-1}p,...,x_{n+1}^{-1}p\in A^*$. This observation prove the equivalence $(i)\Leftrightarrow(ii)$.

$(ii)\Rightarrow(iii)$ Assume that $A$ is not thin. Then there are a finite subset $F$ of $G$ and an injective sequence $(g_m)_{m<\w}$ in $G$ such that $|Fg_m\cap A|>n$. Passing to subsequences of $(g_m)_{m<\w}$, we may suppose that there exist distinct $x_1,...,x_{n+1}\in F$ such that $\{x_1,...,x_{n+1}\}g_m\subseteq A$ so $x_1^{-1}A\cap...\cap x_{n+1}^{-1}A$ is infinite.

$(iii)\Rightarrow(i)$ Assume that $x_1A\cap...\cap x_{n+1}A$ is infinite for some distinct $x_1,...,x_{n+1}\in G$. Then there is an injective sequence $(g_m)_{m<\w}$ in $x_1A\cap...\cap x_{n+1}A$ such that $\{x_1^{-1},...,x_{n+1}^{-1}\}g_m\subset A$ so $A$ is not $n$-thin. \end{proof}
By ~\cite{b7}, a subset $A$ of a countable group $G$ is $n$-thin if and only if $A$ can be partitioned into $\le n$ thin subsets. The following statements are from ~\cite{b15}. Every $n$-thin subset of an Abelian group of cardinality $\aleph_m$ can be partitioned into $\le n^{m+1}$ thin subsets. For each $m\ge2$ there exist a group $G$ of cardinality $\aleph_n$, $n=\frac{m(m+1)}{2}$ and a $2$-thin subset $A$ of $G$which cannot be partitioned into $m$ thin subsets. Moreover, there is a group $G$ of cardinality $\aleph_\w$ and a $2$-thin subset $A$ of $G$ which cannot be finitely partitioned into thin subsets.

Remind that an ultrafilter $p\in G^*$ is strongly prime if $p\in G^*\setminus\overline{G^*G^*}$.
\begin{Ps} ~\label{p5} For a subset $A$ of a group $G$, the following statements are equivalent

$(i)$ $A$ is sparse;

$(ii)$ every ultrafilter $p\in A^*$ is strongly prime;

$(iii)$ for every infinite subset $X$ of $G$, there exists a finite subset $F\subset X$ such that $\bigcap_{g\in F}gA$ is finite. \end{Ps}
\begin{proof} The equivalence $(ii)\Leftrightarrow(iii)$ was proved in ~\cite[Theorem 9]{b3}.

To prove $(i)\Leftrightarrow(ii)$, it suffices to note that $\vt_p(A)$ is infinite if and only if $\vt_p(A)$ has a limit point $qp$, $q\in G^*$ in $A^*$. \end{proof}

\begin{Ps}~\label{p6} A subset $A$ of a group $G$ is sparse if and only if, for every countable subgroup $H$ of $G$, $A\cap H$ is sparse in $H$. \end{Ps}
\begin{proof} Assume that $A$ is not sparse. By Proposition~\ref{p5} $(iii)$, there is a countable subset  $X=\{x_n:n<\w\}$ of $G$ such that for any $n<\w$ $x_0A\cap...\cap x_nA$ is infinite. For any $n<\w$, we pick $a_n\in x_0A\cap...\cap x_nA$, put $S=\{x_0^{-1}a_n,...,x_n^{-1}a_n:n<\w\}$ and denote by $H$ the subgroup of $G$ generated by $S\cup X$. By Proposition~\ref{p5}$(iii)$, $A\cap H$ is not sparse in $H$. \end{proof}
A family $\II$ of subsets of a group $G$ is called an ideal in the Boollean algebra $\PP_G$ of all subsets of $G$ if $A,B\in\II$ implies $A\cup B\in\II$, and $A\in\II$, $A'\subset A$ implies $A'\in\II$. An ideal $\II$ is left (right) translation invariant if $gA\in\II$ ($Ag\in\II$) for each $A\in\II$.
\begin{Ps} The family $Sp_G$ of all sparse subsets of a group $G$ is a left and right translation invariant ideal in $\PP_G$\end{Ps}
\begin{proof} Apply Proposition~\ref{p1}. \end{proof}
\begin{Ps}~\label{p8} For a subset $A$ of a group $G$, the following statements are equivalent

$(i)$ $A$ is disparse;

$(ii)$ if $p\in A^*$ then $p\notin G^*p$. \end{Ps}
Recall that an element $s$ of a semigroup $S$ is right cancelable if, for any $x,y\in S$, $xs=ys$ implies $x=y$.
\begin{Ps}~\label{p9} A subset $A$ of a countable group $G$ is disparse if and only if each ultrafilter $p\in A^*$ is right cancelable in $\beta G$.\end{Ps}
\begin{proof} By ~\cite[Theorem 8.18]{b5}, for a countable group $G$, an ultrafilter $p\in G^*$ is right cancelable in $\beta G$ if and only if $p\notin G^*p$. Apply Proposition~\ref{p8}\end{proof}
\begin{Ps}~\label{p10} The family $dSp_G$ of all disparse subsets of a group $G$ is a left and right translation invariant ideal in $\PP_G$.\end{Ps}
\begin{proof} Assume that $A\cup B$ is not disparse and pick $p\in G^*$ such that $\vt_P(A\cup B)$ has a non-isolated point $gq$. Then either $gp\in A^*$ or $gp\in B^*$ so $gp$ is non-isolated either in $\vt_p(A)$ or in $\vt_p(B)$.

To see that $dSp_G$ is translation invariant, we apply Proposition~\ref{p1}. \end{proof}
For an injective sequence $(a_n)_{n<\w}$ in a group $G$, we denote
$$FP(a_n)_{n<\w}=\{a_{i_1}a_{i_2}...a_{i_n}:i_1<...<i_n<\w\}.$$
\begin{Ps}~\label{p11} For every disparse subset $A$ of a group $G$, the following two equivalent statements hold

$(i)$ if $q$ is an idempotent from $G^*$ and $g\in G$ then $qg\notin A^*$;

$(ii)$ for each injective sequence $(a_n)_{n<\w}$ in $G$ and each $g\in G$, $FP(a_n)_{n<\w}g\setminus A$ is infinite. \end{Ps}
\begin{proof} The equivalence $(i)\Leftrightarrow(ii)$ follows from two well-known facts. By ~\cite[Theorem 5.8]{b5}, for every idempotent $q\in G^*$ and every $Q\in q$, there is an injective sequence $(a_n)_{n<\w}$ in $Q$ such that $FP(a_n)_{n<\w}\subseteq Q$. By ~\cite[Theorem 5.11]{b5}, for every injective sequence $(a_n)_{n<\w}$ in $G$, there is an idempotent $q\in G^*$ such that $FP(a_n)_{n<\w}\in q$.

Assume that $qg\in A^*$. Then $q(qg)=qg$ so $qg\in G^*qg$ and, by Proposition~\ref{p8}, $A$ is not disparse. \end{proof}
\begin{Ps}~\label{p12} For every infinite group $G$, we have the following strong inclusions $$Sp_G\subset dSp_G\subset Sm_G,$$ where $Sm_G$ is the ideal of all small subsets of $G$. \end{Ps}
\begin{proof} Clearly, $Sp_G\subseteq dSp_G$. To verify $dSp_G\subseteq Sm_G$, we assume that a subset $A$ of $G$ is not small. Then $A$ is prethick and, by Proposition~\ref{p3}$(ii)$, there exist $p\in G^*$ and a finite subset $F$ of $G$ such that $\vt_p(FA)=Gp$. Hence, $G^*p\subseteq (FA)^*$. We takean arbitrary idempotent $q\in G^*$ and choose $g\in F$ such that $qp\in (gA)^*$. Since $q(qp)=qp$ so $q\in G^*qp$ and, by Proposition~\ref{p8}$(ii)$, $gA$ is not disparse. By Proposition~\ref{p10} $A$ is not disparse.

To prove that $dSp_G\setminus Sp_G\neq\varnothing$ and $Sm_G\setminus dSp_G\neq\varnothing$, we may suppose that $G$ is countable. We put $F_0=\{e\}$ and write $G$ as an union of an increasing chain $\{F_n:n<\w\}$ of finite subsets.

$1.$ To find a subset $A\in dSp_G\setminus Sp_G$, we choose inductively two sequences $(a_n)_{n<\w}$, $(b_n)_{n<\w}$ in $G$ such that

$(1)$ $F_nb_n\cap F_{n+1}b_{n+1}=\varnothing,\text{ } n<\w$;

$(2)$ $F_ia_ib_j\cap F_ka_kb_m=\varnothing$, $0\le i\le j<\w$, $0\le k\le m<\w$, $(i,j)\neq(k,m)$.\\
We put $a_0=b_0=e$ and assume that $a_0,...,a_n$, $b_0,...,b_n$ have been chosen. We choose $b_{n+1}$ to satisfy $F_{n+1}b_{n+1}\cap F_ib_i=\varnothing$, $i\le n$ and $$\bigcup_{0\le i\le j<\w}F_ia_ab_i\cap(\bigcup_{o\le i\le n}F_ia_i)b_{n+1}=\varnothing.$$
Then we pick $a_{n+1}$ so that $$F_{n+1}a_{n+1}b_{n+1}\cap(\bigcup_{0\le i\le j<\w F_ia_ib_j})=\varnothing, \text{ } F_{n+1}a_{n+1}b_{n+1}\cap(\bigcup_{0\le i\le n} F_ia_ib_{n+1})=\varnothing.$$
After $\w$ steps, we put $A=\{a_ib_j:0\le i\le j<\w\}$, choose two free ultrafilters $p,q$ such that $\{a_i:i<\w\}\in p$, $\{b_i: i<\w\}\in q$ and note that $A\in pq$. By Proposition~\ref{p5}$(ii)$, $A\notin Sp_G$.

To prove that $A\in dSp_G$, we fix $p\in G^*$ and take an arbitrary $q\in\vt_p(A)$. For $n<\w$, let $A_n=\{a_ib_j:0\le i\le n, i\le j<\w\}$. By $(1)$, the set $\{b_j:j<\w\}$ is thin. Applying Proposition~\ref{p3} and Proposition~\ref{p1}, we see that $A_n$ is sparse. Therefore, if $A_n\in q$ for some $n<\w$ then $q$ is isolated in $\vt_p(A)$. Assume that $A_n\notin q$ for each $n<\w$. We take an arbitrary $g\in G\setminus\{e\}$ and choose $m<\w$ such that $g\in F_m$. By $(2)$, $g(A\setminus A_m)\cap A=\emptyset$ so $gq\notin A^*$. Hence, $\vt_p(A)=\{q\}$.

$2.$ To find a subset $A\in Sm_G\setminus dSp_G$, we choose inductively two sequences $(a_n)_{n<\w}$, $(b_n)_{n<\w}$ in $G$ such that, for each $m<\w$, the following statement hold

$(3)$ $b_mFP(a_n)_{n<\w}\cap F_m(FP(a_n)_{n<\w})=\varnothing.$\\
We put $a_0=e$ and take an arbitrary $g\in G\setminus\{e\}$. Suppose that $a_0,...,a_m$ and $b_0,...,b_m$ have been chosen. We pick $b_{m+1}$ so that
$$b_{m+1}FP(a_n)_{n\le m}\cap F_{m+1}(FP(a_n)_{n\le m})=\varnothing$$
and choose $a_{n+1}$ such that
$$b_{m+1}(FP(a_n)_{n\le m})a_{n+1}\cap F_{m+1}(FP(a_n)_{n\le m})=\varnothing,$$
$$b_{m+1}(FP(a_n)_{n\le m})\cap F_{m+1}(FP(a_n)_{n\le m})a_{n+1}=\varnothing.$$
After $\w$ steps, we put $A=FP(a_n)_{n<\w}$. By Proposition~\ref{p11}, $A\notin dSp_G$. To see that $A\in Sm_G$, we use $(3)$ and the following observation. A subset $S$ of a group $G$ is small if and only if $G\setminus FS$ is large for each finite subset $F$ of $G$.
\end{proof}
\begin{Ps}~\label{p13} Let $G$ be a direct product of some family $\{G_\alpha:\alpha<\kappa\}$ of countable groups. Then $G$ can be partitioned into $\aleph_0$ disparse subsets. \end{Ps}
\begin{proof} For each $\alpha<\kappa$, we fix some bijection $f_\alpha: G_\alpha\setminus\{e_\alpha\}\to\NN$, where $e_\alpha$ is the identity of $G_\alpha$. Each element $g\in G\setminus\{e\}$ has the unique representation
$$g=g_{\alpha_1}g_{\alpha_2}...g_{\alpha_n},\text{ }\alpha_1<\alpha_2<...<\alpha_n<\kappa,\text{ }g_{\alpha_i}\in G_{\alpha_i}\setminus\{e_{\alpha_i}\}.$$
We put $supt g=\{\alpha_1,...,\alpha_n\}$ and let $Seq_\NN$ denotes the set of all finite sequence in $\NN$. We define a mapping $f: G\setminus\{e\}\to Seq_\NN$ by $$f(g)=(n,f_{\alpha_1}(g_{\alpha_1}),...,f_{\alpha_n}(g_{\alpha_n}))$$
and put $D_s=f^{-1}(s)$, $s\in Seq_\NN$. 

We fix some $s\in Seq_{\NN}$ and take an arbitrary $p\in G^*$ such that $p\in D_s^*$. Let $s=\{n,m_1,...,m_n\}$, $g\in D_s$ and $i\in supt g$. It follows that, for each $i<\kappa$, there exists $x_i\in G_i$ such that $x_iH_i\in p$, where $H_i=\otimes\{G_i:j<\kappa,j\neq i\}$. We choose $i_1,...,i_k$, $k<n$ such that $$\{i_1,...,i_k\}=\{i<\kappa:x_iH_i\in p,\text{ }x_i\neq e_i\},$$
put $P=x_{i_1}H_{i_1}\cap...\cap x_{i_k}H_{i_k}\cap D_s$ and assume that $gp\in P^*$ for some $g\in G\setminus\{e\}$. Then $supt g\cap\{i_1,...,i_k\}=\varnothing$. Let $supt g=\{j_1,...,j_t\}$, $H=H_{j_1}\cap...\cap H_{j_t}$. Then $H\in p$ but $g(H\cap P)\cap D_s=\varnothing$ because $|supt gx|>n$ for each $x\in H\cap P$. In particular, $gp\notin P^*$. Hence, $p$ is isolated in $\vt_p(D_s)$.\end{proof}

By Proposition~\ref{p13}, every infinite group embeddable in a direct product of countable groups (in particular, every Abelian group) can be partitioned into $\aleph_0$ disparse subsets.
\begin{Qs} Can every infinite group be partitioned into $\aleph_0$ disparse subsets? \end{Qs}

By \cite{b9}, every infinite group can be partitioned into $\aleph_0$ small subsets. For an infinite group $G$, $\eta(G)$ denotes the minimal cardinality $\kappa$ such that $G$ can be partitioned into $\eta(G)$ sparse subsets. By \cite[Theorem 1]{b11}, if $|G|>(\kappa^+)^{\aleph_0}$ then $\eta(G)>\kappa$, so Proposition~\ref{p12} does not hold for partition of $G$ into sparse subsets. For partitions of groups into thin subsets see \cite{b10}.

\section{Comments}~\label{s3}

$1.$ A subset $A$ of an amenable group $G$ is called {\em absolute null} if $\mu(A)=0$ for each Banach measure $\mu$ on $G$, i.e. finitely additive left invariant function $\mu:\PP_G\to[0,1]$. By \cite[Theorem 5.1]{b6} Proposition~\ref{p5}, every sparse subset of an amenable group $G$ is absolute null.

\begin{Qs} Is every disparse subset of an amenable group $G$ absolute null? \end{Qs}
To answer this question in affirmative, in view of Proposition~\ref{p8}, it would be enough to show that each ultrafilter $p\in G^*$ such that $p\notin G^*p$ has an absolute null member $P\in p$. But that is not true. We sketch corresponding counterexample.

We put $G=\ZZ$ and choose inductively an injective sequence $(a_n)_{n<\w}$ in $\NN$ such that, for each $m<\w$ and $i\in\{-(m+1),...,-1,1,...,m+1\}$, the following statements hold

$(\ast)\text{ } (\bigcup_{n>m}(a_n+2^{a_n}\ZZ))\cap(i+\bigcap_{n>m}(a_n+2^{a_n}\ZZ))=\varnothing$

Then we fix an arbitrary Banach measure $\mu$ on $\ZZ$ and choose an ultrafilter $q\in\ZZ^*$ such that $2^n\ZZ\in q$, $n\in\NN$ and $\mu(Q)>0$ for each $Q\in q$. Let $p\in G^*$ be a limit point of the set $\{a_n+q:n<\w\}$. Clearly, $\mu(P)>0$ for each $P\in p$. On the other hand, by $(\ast)$, the set $\ZZ+p$ is discrete so $p\notin\ZZ^*+p$.

In \cite{b17} S. Solecki, for a group $G$, defined two functions $\sigma^R,\sigma^L:\PP_G\to [0,1]$ by the formulas
$$\sigma^R(A)=\inf_F\sup_{g\in G}\frac{F\cap Ag}{|F|},\text{ }\sigma^L(A)=\inf_F\sup_{g\in G}\frac{|F\cap gA|}{|F|},$$
where $\inf$ is taken over all finite subsets of$G$.

By \cite{b1} and \cite{b20}, a subset $A$ of an amenable group is absolute null if and only if $\sigma^R(A)=0$.
\begin{Qs} Is $\sigma^R(A)=0$ for every sparse subset $A$ of a group $G$? \end{Qs}
To answer this question positively it suffices to prove that if $\sigma^R(A)>0$ then there is $g\in G\setminus\{e\}$ such that $\sigma^R(A\cap gA)>0.$

$2.$ The origin of the following definition is in asymptology (see \cite{b16}, \cite{b19}). A subset $A$ of a group $G$ is called {\em asymptotically scattered} if, for any infinite subset $X$ of $A$, there is a finite subset $H$ of $G$ such that, for any finite subset $F$ of $G$ satisfying $F\cap H=\varnothing$, we can find a point $x\in X$ such that $Fx\cap A=\varnothing$. By \cite[Theorem 13]{b13} and Propositions~\ref{p5} and~\ref{p6}, a subset $A$ is sparse if and only if $A$ is asymptotically scattered.

We say that a subset $A$ of $G$ is {\em ultrascattered} if, for any $p\in G^*$, the space $\vt_p(A)$ is scattered, i.e. each subset of $\vt_p(A)$ has an isolated point. Clearly, each disparse subset is ultrascattered.
\begin{Qs} How one can detect whether given a subset $A$ of $G$ is ultrascattered? Is every ultrascattered subset small? \end{Qs}
We say that a subset $A$ of $G$ is {\em weakly asymptotically scattered} if, for any subset $X$ of $A$, there is a finite subset $H$ of $G$ such that, for any finite subset $F$ of $G$ satisfying $F\cap H=\varnothing$, we can find a point $x\in X$such that $Fx\cap X=\varnothing$.
\begin{Qs} Are there any relationships between ultrascattered and weakly asymptotically scattered subsets? \end{Qs}

$3.$ Let $A$ be a subset of a group $G$ such that each ultracompanion $\vt_p(A)$ is compact. We show that $A$ is sparse. In view of Proposition~\ref{p6}, we may suppose that $G$ is countable. Assume the contrary: $\vt_p(A)$ is infinite for some $p\in G^*$. On one hand, the countable compact space $\vt_p(A)$ has an injective convergent sequence. On the other hand, $G^*$ has no such a sequence.

$4.$ Let $X$ be a subset of a group $G$, $p\in G^*$. We say that the set $Xp$ is uniformly discrete if there is $P\in p$ such that $xP^*\cap yP^*=\varnothing$ for all distinct $x,y\in X$.
\begin{Qs} Let $A$ be a disparse subset of a group $G$. Is $\vt_p(A)$ uniformly discrete for each $p\in G^*$? \end{Qs}

$5.$ Let $\FF$ be a family of subsets of a group $G$, $A$ be a subset of $G$. We denote $\delta_{\FF}(A)=\{g\in G:gA\cap A\in \FF\}$. If $\FF$ is the family of all infinite subsets of $G$, $\delta_{\FF}(A)$ were introduced in \cite{b14} under name combinatorial derivation of $A$. Now suppose that $\delta_p(\FF)\neq\varnothing$, pick $q\in A^*\cap Gp$ and note that $\delta_p(A)=\delta_q(A)$. Then $\delta_q(A)=(\delta_q(A))^{-1}q$.


\begin{thebibliography}{99}
\bibitem{b1} T.~Banakh, {\it The Solecki submeasures and densities on groups}, preprint available at arXiv:1211.0717.
\bibitem{b2} H.~Dales, A.~Lau, D.~Strauss, {\it Banach Algebras on semigroups and their compactifications}, Mem. Amer. Math. Soc. 2005 (2010).
\bibitem{b3} M.~Filali, Ie.~Lutsenko, I.~Protasov, {\it Boolean group ideals and the ideal structure of $\beta G$}, Math. Stud. {\bf 30}, (2008), 1-10.
\bibitem{b4} M.~Filali, I.~Protasov, {\it Ultrafilters and Topologies on Groups}, Math. Stud. Monogr. Ser., Vol. 13, VNTL Publisher. Lviv, 2010.
\bibitem{b5} N.~Hindman, D.~Strauss, {\it Algebra in the Stone-$\check{C}$ech Comactification}, 2nd edition, de Grueter, 2012.
\bibitem{b6} Ie.~Lutsenko, I.V.~Protasov, {\it Sparse, thin and other subsets of groups}, Intern. J. Algebra Computation, {\bf 19} (2009), 491--510.
\bibitem{b7} Ie.~Lutsenko, I.V.~Protasov, {\it Thin subsets of balleans}, Appl. Gen. Topology {\bf 11} (2010), 89--93.
\bibitem{b8} Ie.~Lutsenko, I.V.~Protasov, {\it Relatively thin and sparse subsets of groups}, Ukr. Math. J. {\bf 63} (2011), 216--225.
\bibitem{b9} I.~Protasov, {\it Small systems of generators of groups}, Math. Notes {\bf 76} (2004), 420--426.
\bibitem{b10} I.~Protasov, {\it Partitions of groups into thin subsets}, Algebra Discrete Math. {\bf 11, 1} (2011), 88-92.
\bibitem{b11} I.~Protasov, {\it Partitions of groups into sparse subsets}, Algebra Discrete Math. {\bf 13, 1} (2012), 107-110.
\bibitem{b12} I.~Protasov, {\it Selective survey on subset combinatorics of groups}, J. Math. Sciences {\bf 174} (2011), 486-514.
\bibitem{b13} I.V.~Protasov, {\it Asymptotically scattered spaces}, preprint available at arXiv:1212.0364.
\bibitem{b14} I.V.~Protasov, {\it Combinatorial derivation}, Appl. Gen. Topology (to appear), preprint available at arXiv:1210.0696.
\bibitem{b15} I.V.~Protasov, S. Slobodianiuk, {\it Thin subsets of groups}, Ukr. Math. J. (to appear)
\bibitem{b16} I.~Protasov, M.~Zarichnyi, {\it General Asymptology}, Math. Stud. Monogr. Ser., Vol. 12, VNTL Publishers, Lviv, 2007
\bibitem{b19} J.~Roe, {\it Lectures on Coarse Geometry}, Amer. Math. Soc., Providence RI, 2003.
\bibitem{b17} S.~Solecki, {\it Size of subsets and Haar null sets}, Geom. Funct. Analysis, {\bf 15} (2005), 246--273.
\bibitem{b18} S.~Todorcevic, {\it Introduction to Ramsey Spaces}, Princeton Univ. Press, 2010.
\bibitem{b20} P.~Zakrzewski, {\it On the comlexity of the ideal of absolute null sets}, Ukr. Math. J. {\bf 64} (2012), 306--308.
\bibitem{b21} Y.~Zelenyuk, {\it Ultrafilters and Topologies on Groups}, de Grueter, 2012.
\end{thebibliography}
\end{document}